\documentclass[a4paper,reqno]{amsart} 

\usepackage{amssymb}
\usepackage[mathcal]{euscript} 
\usepackage[cmtip,all]{xy}
\usepackage{color}
\usepackage{tikz-cd} 

\newcommand{\bydef}{:=}

\newcommand{\id}{\mathrm{id}}
\newcommand{\ex}{\mathrm{ex}}

\newcommand{\cA}{\mathcal{A}}
\newcommand{\cB}{\mathcal{B}}

\newcommand{\cH}{\mathcal{H}}

\newcommand{\cK}{\mathcal{K}}
 
\newcommand{\cM}{\mathcal{M}}

\newcommand{\cS}{\mathcal{S}}

\DeclareMathOperator{\CD}{\mathfrak{CD}}

\newcommand{\ZZ}{\mathbb{Z}}

\newcommand{\FF}{\mathbb{F}} 

\DeclareMathOperator{\AlgF}{\mathrm{Alg_{\FF}}}

\DeclareMathOperator{\Aut}{\mathrm{Aut}}

\DeclareMathOperator{\Der}{\mathrm{Der}}

\newcommand{\subo}{\bar 0} 
\newcommand{\subuno}{\bar 1}

\newcommand{\invol}{\,-\,}

\newtheorem{theorem}{Theorem}[section]
\newtheorem{proposition}[theorem]{Proposition}
\newtheorem{lemma}[theorem]{Lemma}

\theoremstyle{definition} 

\newtheorem{example}[theorem]{Example}

\theoremstyle{remark} 
\newtheorem{remark}[theorem]{Remark}
\numberwithin{equation}{section}

\begin{document}
	
	\title{Automorphisms, derivations and gradings of the split quartic cayley algebra.}

		\author[V.~Blasco]{Victor ~Blasco}
			\thanks{}
		
			\address{Departamento de
			Matem\'{a}ticas e Instituto Universitario de Matem\'aticas y
			Aplicaciones, Universidad de Zaragoza, 50009 Zaragoza, Spain}
		\email{victorblasco98@gmail.com} 

		\author[A.~Daza-Garcia]{Alberto
		~Daza-Garcia} 
		\thanks{The second author acknowledges support by the F.P.I. grant PRE2018-087018. He also aknowledges support by grant MTM2017-83506-C2-1-P (AEI/FEDER, UE) and by grant PID2021-123461NB-C21, funded by MCIN/AEI/10.13039/501100011033 and by "ERDF A way of making Europe".}
	\address{Departamento de
		Matem\'{a}ticas e Instituto Universitario de Matem\'aticas y
		Aplicaciones, Universidad de Zaragoza, 50009 Zaragoza, Spain}
	\email{albertodg@unizar.es}

	
	
	
	\begin{abstract} The split quartic Cayley algebra is a structurable algebra which has been used to give constructions of Lie algebras of type D4. Here, we calculate it's group of automorphisms, it's algebra of derivations and it's gradings. \end{abstract}
	
	
	\maketitle
	
	\section{Introduction}\label{se:intro}
	
	Structurable algebras are a class of algebras with involution introduced by Allison in 1978 \cite{AliStr} as a generalization of Jordan algebras. They are a generalization in the sense that they also have a Tits-Kantor-Koecher (TKK) construction of a Lie algebra. One of these algebras is the split quartic Cayley algebra which is used for example in \cite{AliD4} to give constructions of Lie algebras of type D4. Here, we calculate its group of automorphism, the algebra of derivations and it's gradings up to isomorphism.

    The structure is as follows: in section 2 we define the split quartic Cayley algebra and give a multiplication table, in section 3 we calculate it's group of automorphisms and it's algebra of derivations and in section 4 we calculate it's automorphisms.

	We are going to work over an algebraically closed field $\FF$ of characteristic different from $2,3$ and $5$. Groups are going to be considered abelian and it's neutral element will be denoted by $e$, unless we work with specific groups with their own notation.
	
    \section{The split quartic Cayley algebra}
	
	This section is devoted to introduce the split Cayley algebra. In order to do so, we recall a modified Cayley Dickson process introduced in \cite{AliFauCD} starting with the algebra $\cB=\FF\oplus \FF\oplus \FF\oplus \FF$. Take $\mu\in \FF^{\times}$. We denote by $t$ the trace of $\cB$. Define $b^{\theta}=-b+\frac{1}{2}t(b)1$ for all $b\in \cB$. Let $\cA=\cB\oplus s\cB=\{b_1+sb_2\mid b_1,b_2\in \cB\}$. We define a product and an involution in $\cA$ by:
	
	\begin{equation*}
		\begin{split}
			(b_1+sb_2)(b_3+sb_4)=(b_1b_3+\mu(b_2b_4^{\theta})^{\theta})+s(b_1^{\theta}b_4+(b_2^{\theta}b_{3}^{\theta})^{\theta})\\
			\overline{b_1+sb_2}=b_1-sb_2^{\theta}
		\end{split}
	\end{equation*}

	We call this algebra $\CD(\cB,\mu)$. Notice, that since we are in an algebraically closed field, the morphism $b_1+sb_2\mapsto b_1+\sqrt{\mu}sb_2$ is an isomorphism from $\CD(\cB,\mu)$ to $\CD(\cB,1)$. Hence, from now on, we are going to work with the algebra $\CD(\cB,1)$. We call this algebra the \textbf{split quartic Cayley algebra} check with the isomorphism in \cite[Proposition 6.5]{AliFauCD} and the definition in \cite{AliSkew1}.
	
	Call $x_1=(1,1,-1,-1), x_2=(1,-1,1,-1), x_3=(1,-1,-1,1)$. Call $\cK=\FF1\oplus \FF s$ which is a subalgebra of $\cA$ isomorphic to $\FF\times \FF$ via the automorphism given by $1\mapsto (1,1)$, $s\mapsto (1,-1)$. Then, the action $\circ\colon\cK\times \cA\to \cA$ given by $g\circ x=xg$ for all $g\in\cK,x\in \cA$, endows $\cA$ with a structure of left $\cK$-module, which is a free $\cK$-module spanned by $1,x_1,x_2,x_3$. If we identify $\cK$ with $\FF\times \FF$ and call $\ex$ the involution given by $\ex(x,y)=(y,x)$,the multiplication and the involution follows from the following rules:
	
	\begin{equation}\label{eq:product}
		\begin{split}
		(f1)(g1)=(fg)1, \  
		(gx_i)(f1)&=(fg)x_i=(\overline{f}1)(fx_i)\\
		(fx_i)(gx_i)=(f\overline{g})1, \ 
		(f x_i)(gx_j)&=(\overline{fg}x_k)
		\end{split}
\end{equation}
for all $f,g\in \cK$ and $\{i,j,k\}=\{1,2,3\}$ and
\begin{equation}\label{eq:invol}
		\overline{f_01+f_1x_1+f_2x_2+f_3x_3}=\ex(f_0)1+f_1x_1+f_2x_2+f_3x_3
	\end{equation}
for all $f_0,f_1,f_2,f_3\in\cK$.

\begin{remark}\label{rm:subspaces}
	Notice that if we define the subspaces $\cS=\{x\in\cA\mid \overline{x}=-x\}, \cH=\{x\in \cA\mid \overline{x}=x\}, \cM=\{x\in \cH\mid sx+xs=0\}$, we get that:
	\begin{equation*}
			\cS=\FF s, \ \cH=\FF1\oplus \left(\bigoplus_{i=1}^3\cK x_i\right), \  \cM=\bigoplus_{i=1}^3\cK x_i, \ \AlgF(\cS)=\cK
	\end{equation*}
\end{remark}

\begin{remark}\label{rm:StandardGrading}There is a $\ZZ_2^2$ grading of $\cA$ given by $\cA_{(\subo,\subo)}=\cK$, $\cA_{(\subo,\subuno)}=\cK x_1$, $\cA_{(\subuno,\subo)}=\cK x_2$ and $\cA_{(\subuno,\subuno)}=\cK x_3$. We call this grading the \textbf{standard quartic grading} and denote it by $\Gamma_{SQ}$.
\end{remark}

\section{Automorphisms and derivations}

In this section we calculate the groups of automorphisms and the algebra of derivations of $(\cA,\invol)$ (i.e. those automorphisms and derivations which commute with the involution). We begin with some easy properties:

\begin{lemma}\label{l:AutomorphismProp} Let $\varphi\in\Aut(\cA,\invol)$ and $d\in \Der(\cA,\invol)$
	\begin{itemize}
		\item[(1)] $\varphi(\cS)= \cS$, $\varphi(\cH)=\cH$, $d(\cS)\subseteq \cS$ and  $d(\cH)\subseteq\cH$
				\item[(2)] $\varphi(\cK)= \cK$, $\varphi(\cM)=\cM$, $d(\cK)=0$ and $d(\cM)\subseteq \cM$.
		\end{itemize}
\end{lemma}
		\begin{proof}
			\begin{itemize}Each conteinment '$\subseteq$' in $(1)$ is due to the fact that the involution commutes with $\varphi$ and $d$. The equalities follow from the fact that $\varphi$ is invertible.
	
			Since $\varphi(\cS)=\cS$, There is $\lambda\in \FF^{\times}$ such that $\varphi(s)e=\lambda s$. Since $\varphi(1)=1$, we get $\varphi(\cK)= \cK$. If $m\in \cH$ and $sm+ms=0$, applying $\varphi$ we get that $\lambda(s\varphi(m)+\varphi(m)s)=0$. Hence $\varphi(\cM)\subseteq\cM$. We get the equality since $\varphi$ is invertible. Since $d$ is a derivation $d(1)=0$. Using $(1)$, there is $\beta$ such that $d(s)=\beta s$. Since $0=d(1)=d(s^2)=2\beta 1$, we get that $d(\cK)=0$. Finally, if $m\in \cM$ $0=d(sm+ms)=sd(m)+d(m)s$. Using $(1)$ it follows  $d(\cM)\subseteq \cM$.
		\end{itemize}
		\end{proof}

Now, we will start calculating the automorphisms. In order to do so, we let $S_3$ be the symmetric on $3$ elements, and we will need the following lemma.

\begin{lemma}\label{l:PreservesStandardGrad}
	Let $\varphi\in\Aut(\cA,\invol)$. There is a permutation in $S_3$ which we denote $\sigma_{\varphi}$ such that $\varphi(\cK x_i)=\cK x_{\sigma_{\varphi}(i)}$ for all $i\in \{1,2,3\}$.
\end{lemma}

\begin{proof}
Due to lemma \ref{l:AutomorphismProp} there are $r_1,r_2,r_3\in\cK$ such that $\varphi(x_i)=r_1x_1+r_2x_2+r_3x_3$. Let $i$ be such that $r_i\neq 0$. Then since $1=\varphi(x_i)^2= r_1\overline{r_1}+r_3\overline{r_3}+r_3\overline{r_3}+\overline{r_2r_3}x_1+\overline{r_1r_3}x_2+\overline{r_1r_2}x_3$. That, due to remark \ref{rm:StandardGrading} means that $r_1r_2=r_2r_3=r_3r_1=0$ since up to scalar, the only zero divisors in $\FF\times \FF$ are $(1,0)$ and $(0,1)$, this implies that $r_j,r_k=0$ for $\{i,j,k\}=\{1,2,3\}$. Hence $\varphi(x_1)=r_ix_i$ and $r_i\overline{r_i}=1$. Since due to lemma \ref{l:AutomorphismProp} $\varphi(\cK x_i)=\varphi(\cK)\varphi(x_i)=\cK r_i x_i$, then we have proved that there is a map $\sigma_{\varphi}\colon \{1,2,3\}\to \{1,2,3\}$ such that $\varphi(\cK x_i)=\cK x_{\sigma(i)}$ for all $i\in \{1,2,3\}$. Since $\varphi$ is invertible this map is a permutation.
\end{proof}	

\begin{remark}\label{rm:PermutAut}
Let $\sigma$ be a permutation in $S_3$. We denote by $f_{\sigma}\colon \cA \to \cA$  the map defined as $f_{\sigma}(r_01+r_1x_1+r_2x_2+r_3x_3)=r_01+r_1x_{\sigma(1)}+r_2x_{\sigma(2)}+r_3x_{\sigma(3)}$.  Using \eqref{eq:product} is not hard to check that this is an automorphism of $(\cA,\invol)$. Moreover the map $\theta\colon S_3\to \Aut(\cA,\invol)$ defined by $\sigma\mapsto f_{\sigma}$ is a monomorphism of groups and we denote it's image by $H$.
\end{remark}

If we have an algebra with involution $(\cB,\invol)$, a group $G$ and a grading $\Gamma\colon \cB=\bigoplus_{g\in G}\cB_g$, we denote $\Aut(\cB,\Gamma,\invol)\bydef\{\varphi\in\Aut(\cB,\invol)\mid \varphi(\cB_g)=\cB_g\ \forall g\in G\}$,

\begin{lemma}\label{l:structAut}
	$\Aut(\cA,\invol)\cong \Aut(\cA,\Gamma_{SQ},\invol)\rtimes H$
\end{lemma}
\begin{proof}
	Let $\varphi\in \Aut(\cA,\invol)$. We are going to show that $\varphi\circ f^{-1}_{\sigma_{\varphi}}\in \Aut(\cA,\Gamma_{SQ},\invol)$. Since $\theta$ as defined in Remark \ref{rm:PermutAut} is an automorphism $f^{-1}_{\sigma_{\varphi}}=f_{\sigma_{\varphi}^{-1}}$. By definition $\varphi(\cK x_{i})=\cK x_{\sigma_{\varphi}(i)}$ for all $i$. Hence, $\varphi\circ f^{-1}_{\sigma_{\varphi}}\in \Aut(\cA,\Gamma_{SQ},\invol)$. Therefore, $\Aut(\cA,\invol)= \Aut(\cA,\Gamma_{SQ},\invol)H$. Finally, $f_{\sigma}\in\Aut(\cA,\Gamma_{SQ},\invol)$ if and only if $\sigma=\id$. therefore $\Aut(\cA,\Gamma_{SQ},\invol)\cap H=\{\id\}$ finally it is not hard to show that $\Aut(\cA,\Gamma_{SQ},\invol)$ is a normal subgroup so the result follows.
	
\end{proof}

We denote by $S^1$ the subgroup of $\cK^{\times}$ whose underlying set is $\{r\in\cK^{\times}\mid r\overline{r}=1\}$ and we denote by $C_2$ the ciclic group of order $2$ generated by $\sigma$. We can define an action on $\cK$ by $\sigma(s)=-s$. Like this we identify $C_2$ with $\Aut(\cK)$.
\begin{lemma}\label{l:AutGrad}
$\Aut(\cA,\Gamma_{SQ},\invol)\cong (S^1\times S^1)\rtimes C_2$ with product given by $(r_1,r_2,g)\star(s_1,s_2,h)=(r_1g(s_1),r_2g(s_2),gh)$.
\end{lemma}
\begin{proof}
Consider the morphism $\theta\colon (S^1\times S^1)\rtimes Aut(\cK)\to \Aut(\cA,\Gamma_{SQ},\invol)$ given by $\theta(r_1,r_2,\psi)(s_0+s_1x_1+s_2x_2+s_3x_3)=\psi(s_0)+\psi(s_1)(r_1x_1)+\psi(s_2)(r_2x_2)+\psi(s_3)(r_3x_3)$. Where $r_3=\overline{r_1r_2}$. Using \eqref{eq:product} and \eqref{eq:invol} it is clear that $\theta(r_1,r_2,\psi)$ is an automorphism. Since $r_i(\overline{r}_ix_i)=x_i$ we can check that $\theta(r_1,r_2,\psi)^{-1}=\theta(\psi^{-1}(\overline{r}_1),\psi^{-1}(\overline{r}_2),\psi^{-1})$.

Clearly $\theta$ is injective. Moreover, if $\varphi$ is an element of $\Aut(\cA,\Gamma_{SQ},\invol)$, then, let $\psi=\varphi_{\mid \cK 1}$, $\varphi(x_1)=r_1x_1$,  $\varphi(x_2)=r_2x_2$ and $\varphi(x_3)=r_3x_3$. Since $x_i^2=1$ for $i=1,2$, we get that $r_i\overline{r_i}=1$. Since $x_1x_2=x_3$ we get that $r_3=\overline{r_1r_2}$. Hence, it's easy to show that $\theta(r_1,r_2,\psi)=\varphi$. Since $\Aut(\cK)$ consist on the identity and the involution sending $s$ to $-s$, it's easy to check that it is isomorphic  to $C_2$. \end{proof}

We can finish calculating the automorphisms with the following proposition:

\begin{theorem}
$\Aut(\cA,\invol)\cong ((S^1\times S^1)\rtimes C_2)\rtimes S_3$
\end{theorem}
\begin{proof}
	This is a consequence of Lemma \ref{l:structAut} and Lemma \ref{l:AutGrad}
\end{proof}

Finally, we calculate the derivations. In order to do so, for two given numbers $\lambda,\beta\in\FF$ we define the map $d_{(\lambda,\beta)}\colon \cA\to \cA$ by $d(r_0+r_1x_1+r_2x_2+r_3x_3)=\lambda r_1(sx_1)+\beta r_2(sx_2)-(\lambda+\beta)r_3(sx_3)$.

\begin{theorem}
$\Der(\cA,\invol)=\{d_{\lambda,\beta}\mid \lambda,\beta\in \FF\}$
\end{theorem}
\begin{proof}
	From Lemma \ref{l:AutomorphismProp} we get that for any $d\in \Der(\cA,\invol)$, $d(\cM)\subseteq \cM$. Therefore, for ${i,j,k}=\{1,2,3\}$ we get that $d(x_i)=r_ix_i+r_jx_j+r_kx_k$ for some $r_i,r_j,r_k\in\cK$. Since $0=d(1)=d(x_i^2)=x_id(x_i)+d(x_i)x_i$ we get that $r_i+\overline{r_i}+2(\overline{r_j}x_k+\overline{r_k}x_j)=0$. Therefore, there is some $\lambda_i\in \FF$ such that $d(x_i)=\lambda_i(sx_i)$. Moreover, since $\lambda_3(sx_3)=d(x_3)=d(x_1x_2)=x_1d(x_2)+d(x_1)x_2$, it follows that $\lambda_3=-\lambda_1-\lambda_2$. Finally, since $d(\cK)=0$ and using the properties of the derivations, it follows that $d(r_0+r_1x_1+r_2x_2+r_3x_3)=\lambda_1 r_1(sx_1)+\lambda_2 r_2(sx_2)-(\lambda_1+\lambda_2)r_3(sx_3)$. Now, checking that $d_{1,0}$ and $d_{0,1}$ are derivations is easy and since they span $\{d_{\lambda,\beta}\mid \lambda,\beta\in \FF\}$ we get the equality.

\end{proof}

\section{Gradings}

Given an algebra with involution $(\cA,\invol)$ and a group $G$, a \textbf{$G$-grading} $\Gamma$ on $\cA$ is a vector space decomposition:

\[ \Gamma\colon \cA=\bigoplus_{g\in G}\cA_g\]

satisfying that $\cA_g\cA_h\subseteq\cA_{gh}$ and $\overline{\cA_g}\subseteq \cA_g$ for all $g,h\in G$. If the grading is fixed we refer to $\cA$ as a \textbf{$G$-graded algebra with involution}. We say that an element $x$ is \textbf{homogeneous} if there is some $g\in G$ such that $x\in\cA_g$. In this case we say that $x$ has degree $g$ and we denote it as $\deg(x)=g$. We say that a subspace $V$ of $\cA$ is a \textbf{graded subspace} if $V=\bigoplus_{g\in G}(V\cap \cA_g)$ in this case we will denote $V_g=V\cap \cA_g$.

\begin{remark}\label{rm:GradedHV}
	For a $G$ grading $\Gamma$, $\cS$ and $\cH$ are graded subspaces (see \cite[lemma 3.8]{AC20}). Moreover, since $\cS=s\FF$ and $s^2=1$, we have that $\deg(s)^2=e$ where $e$ is the neutral element of $G$.
\end{remark}
Given two $G$-graded algebras with involution $(\cA,\invol)$ and $(\cB,\invol)$ we say that they are \textbf{isomorphic} if there exist an isomorphism of algebras with involution $\varphi\colon \cA\to \cB$ satisfying that $\varphi(\cA_g)=\cB_g$. 

Given a $G$-grading $\Gamma$ and a $H$-grading $\Gamma'$  of $(\cA,\invol)$ we say that $\Gamma'$ is a \textbf{coarsening} of $\Gamma$ (or that $\Gamma$ is a \textbf{refinement} of $\Gamma'$) if for every $h\in H$ there is a $g\in G$ such that $\cA_g\subseteq \cA_h$.

The basic facts about gradings can be found in \cite{EKmon}.

\begin{example}\label{ex:StructurableGrading}
	Given the split quartic Cayley algebra $(\cA,\invol)$ and $\{i,j,k\}=\{1,2,3\}$ we can define the $\ZZ_2$-grading $\Gamma^i_S\colon \cA=\cA_{\subo}\oplus \cA_{\subuno}$ with $\cA_{\subo}=\cK\oplus \cK x_i$ and $\cA_{\subuno}=\cK x_j\oplus\cK x_k$. 
	
	These gradings are a coarsening of the standard quartic grading $\Gamma_{SQ}$. Moreover, given $i\neq j$ and a permutation $\sigma$ with $\sigma(i)=j$ we get that $\Gamma^i_s$ is isomorphic to $\Gamma_j$ via the automorphism $f_{\sigma}$ with the notation of \ref{rm:PermutAut}
\end{example}

In this section we are going find up to isomorphism the gradings on $(\cA,\invol)$. We start with a lemma:

\begin{lemma}\label{l:KandMGraded}
	For a $G$-grading $\Gamma\colon \cA=\bigoplus_{g\in G}\cA_g$ on $(\cA,\invol)$, the subspaces $\cK$ and $\cM$ are graded subspaces.
\end{lemma}

\begin{proof}
In any algebra with involution $\cS$ and $\FF 1$ are graded subspaces. Hence, $\cK=\FF1\oplus \cS$ is a graded subspace.

Let $m\in \cM$ and let $a_g\in \cA_g$ be such that $m=\sum_{g\in G}a_g$. Let $g_0$ be the degree of $s$ and for every $g\in G$ denote by $\pi_g$ the projection on $\cA_g$ with respect to the decomposition given by the grading. Since $0=sm+ms$ and $0=\pi_{g}(sm+ms)=sa_{g}+a_{g}s$, we get that for every $g\in G$, $a_g\in \cM$. Therefore, $\cM$ is a graded subspace.
\end{proof}

Since $s^2=1$, it's easy to deduce using \eqref{eq:product} that for any $m\in \cM$ $s(sm)=m$. Therefore, we can define two subspaces of $\cM$:

\[\cM_{\sigma}=\{m\in\cM\mid sm=\sigma m\} \ \text{ for }\sigma=\pm\]

And $\cM=\cM_+\oplus\cM_-$.

\begin{lemma}\label{l:MpmGraded}
	For a $G$ grading $\Gamma$, $\deg(s)=e$ if and only if $\cM_+$ and $\cM_-$ are graded subspaces. 
\end{lemma}
\begin{proof}
	Let $\deg(s)=e$. In this case, if $m\in \cM_g$, for some $g\in G$ then, there are $m_+\in\cM_+$ and $m_-\in\cM_-$ such that $m=m_++m_-$. Since $sm=m_+-m_-\in \cM_g$ we get that $m_\sigma\frac{1}{2}(m+(\sigma sm))\in\cM_g$ for $\sigma=\pm$. Hence, $\cM_g=(\cM_+\cap \cM_g)\oplus (\cM_-\cap \cM_g)$. From that is easy to check that $\cM_+$ and $\cM_-$ are graded. If $\cM_+$ and $\cM_-$ are graded, let $g=\deg(s)$. Then, let $m\in(\cM_+)$ for some $h\in G$ it should happen that $h=\deg(m)=\deg(sm)=gh$ and so $g=e$.
\end{proof}

\begin{remark}
	Notice that $\cM_+=\frac{1}{2}(1+s)\cM$ and $\cM_-=\frac{1}{2}(1-s)\cM$ so we are going to call $e_+=\frac{1}{2}(1+s)$ and $e_-=\frac{1}{2}(1-s)$.
\end{remark}

We denote as $b\colon \cM\times \cM\to \FF$ the bilinear form which satisfies $xy=b(x,y)1+\lambda s+m$ for $\lambda\in \FF$ and $m\in \cM$. 

\begin{lemma}\label{l:bilinearForm}
	For any $G$ grading on $\cA$, $b$ is a non-degenerate homogeneous bilinear form (i.e. $b(\cM_g,\cM_h)=0$ if and only if $gh=e$).
\end{lemma}

\begin{proof}
In order to show that $b$ is non degenerate, we take $m\in\cM$. then, there are $r_1,r_2,r_3\in \cK$ such that $m=r_1x_1+r_2x_2+r_3x_3$. Without loss of generality, we suppose that $r_1\neq 0$. Then, either $r_1= \beta e_{\sigma}$ for $\sigma=\pm$ or $r_1\overline{r_1}=\beta x_1$ in both cases with $\beta\neq 0$. In the first case $b(x, e_{-\sigma}x_1)=\beta$ and in the second case $b(x,r_1x_1)=\beta$.

In order to show you that it is homogeneous, we take $x\in\cM_g$ and $y\in \cM_h$. Then, $xy\in\cA_{gh}$. Suppose that $gh\neq e$. If $\deg(s)=gh$, then, $\cA_{gh}=\FF s\oplus \cM\cap \cA_{gh}$ and in other case $\cA_{gh}=\cM\cap \cA_{gh}$

\end{proof}

\begin{example}\label{ex:Gradings} Let $G$ be an abelian group, $i\in \FF$ such that $i^2=-1$ and $\zeta\in \FF$ a primitive cubic root of unit. 
	\begin{itemize}
		\item[(1)]For $g_1,g_2\in G$ denote by $\Gamma_{SQ}(G,g_1,g_2)$ the grading on $(\cA,\invol)$ given by $\deg(s)=e$, $\deg(e_+x_1)=g_1$, $\deg(e_+x_2)=g_2$, $\deg(e_+x_3)=(g_{1}g_2)^{-1}$, $\deg(e_-x_1)=g_1^{-1}$, $\deg(e_- x_2)=g_2^{-1}$ and $\deg(e_-x_3)=g_1g_2$.
		\item[(2)] For $g,g_1,g_2\in G$ with $g$ an element of order $2$, denote by $\Gamma_{SQ}(G,g,g_1,g_2)$ the grading given by $\deg(s)=g$, $deg(x_1)=g_1$, $deg(x_2)=g_2$, $deg(x_3)=g_1g_2$, $deg(sx_1)=gg_1$, $deg(sx_2)=gg_2$ and $deg(sx_3)=gg_1g_2$ 
		\item[(3)] For $\lambda\in\FF^{\times}$ and $h,g,f\in G$ such that $g^2=f^2=h^{-1}$ and $g\neq f$, we denote by $\Gamma_{S}(G,\lambda,h,g,f)$ the grading in which $\deg(s)=e$, $\deg(e_+x_1)=h=\deg(e_-x_1)^{-1}$, $\deg(e_+(x_2+\lambda x_3))=g$, $\deg(e_+(-\lambda^{-1}x_2+x_3))=f$
		\item[(4)] For $h,g\in G$ with $h$ of order $2$, we denote by $\Gamma^1_{S}(G,h,g)$ the grading induced by $\deg(s)=h$, $\deg(e_+x_2+e_-x_3)=g$ and $\deg(e_-x_2+e_+x_3)=g^{-1}$.
		\item[(5)] For  $h,g\in G$ such that $h$ has order $2$ and $g$ has order $4$ we denote by $\Gamma^2_{S}(G,h,g)$ the grading for which $\deg(s)=h$, $\deg(x_1)=g^2$ and $\deg(x_2+ix_3)=g$.
		\item[(6)] For $h,g,f\in G$ with $h,g$ and $f$ of order $2$, we denote by  $\Gamma^3_{S}(G,h,g,f)$  the grading for which $\deg(s)=h$ and $\deg(x_2+x_3)=g$ and $\deg(x_2-x_3)=f$.
		\item[(6)] For $g_1,g_2\in G$ of order $3$ and $g_1\neq g_2 \neq (g_1g_2)^{-1}$, we denote by $\Gamma(G,g_1,g_2)$ the grading given by $\deg(s)=e$, $\deg(e_+(x_1+\zeta x_2+\zeta^2 x_3))=g_1$, $\deg(e_+(x_1+\zeta^2x_2+\zeta x_3))=g_2$, $\deg(e_+(x_1+x_2+x_3))=(g_1g_2)^{-1}$
		\item[(7)] For $h,g_1\in G$ such that $h$ has order $2$ and $g$ has order $3$, we denote by $\Gamma(G,h,g)$ the grading given by $\deg(s)=h$ and $\deg(x_1+\zeta x_2+\zeta^2 x_3)=g$.
	\end{itemize} 
\end{example}

Given a $G$-grading $\Gamma \colon \cA=\bigoplus_{g\in G}\cA_g$ and a $H$-grading $\Gamma' \colon \cA=\bigoplus_{h\in H}\cA_h$, we say that the gradings are \textbf{compatible} if $\cA=\bigoplus_{(g,h)\in G\times H}\cA_g\cap \cA_h$.

\begin{proposition}\label{r:GradingsTypeI}
	If $\Gamma$ is a grading is compatible with $\Gamma_{SQ}$, then it is isomorphic to either $\Gamma_{SQ}(G,g_1,g_2)$ for some $g_1,g_2\in G$ as in example \ref{ex:Gradings}or to $\Gamma_{SQ}(G,g,g_1,g_2)$ for some $g,g_1,g_2$ as in example \ref{ex:Gradings}. 
\end{proposition}
\begin{proof}
	If $\Gamma\colon\cA=\bigoplus_{g\in G}\cA_g$ is such a grading  with $\deg(s)=e$. Since it is compatible with $\Gamma_{SQ}$, for every $i=1,2,3$ there should be a $g_i$ such that $\cA_{g_i}\cap \cK x_i= \cK x_i$ or ${g_i}$ and ${g_i}'$ such that $(\cA_{g_i}\cap\cK x_i)\oplus(\cA_{g'_i}\cap \cK x_i)$. In the first case, $\deg(e_+x_i)=\deg(e_-x_i)=g_i$. In the second case, since $s(\cA_{h}\cap \cK x_i)=(\cA_{h}\cap \cK x_i)$ for $h=g_i$ or $h=g'_i$, we can assume that $\cA_{g_i}\cap \cK x_i=e_+x_i$ and that $\cA_{g'_i}\cap \cK x_i=e_-x_i$. Since $(e_+x_i)(e_-x_i)=e_+$, we get that in both cases $\deg(e_+x_i)\deg(e_-x_i)=e$. Finally, since $(e_-x_1)(e_-x_2)=e_+x_3$, we get that $g_3=g_1^{-1}g_2^{-1}$. Hence $\Gamma=\Gamma_{SQ}(G,g_1,g_2)$.
	
	If $\Gamma\colon\cA=\bigoplus_{g\in G}\cA_g$ is such a grading  with $\deg(s)=g$ for $g$ an order $2$ element. Since each $\cK x_i$ are graded, if for $\sigma=\pm$, $e_{\sigma}x_i$ is homogeneous, $\deg(e_{\sigma}x_i)=\deg(s(e_{\sigma}x_i))=g\deg(e_{\sigma}x_i)$. Hence, for every $i=1,2,3$ there is a group element $g_i$ there is an invertible $r_i\in \cK$ such that $\deg(r_ix_i)=g_i$ since the field is algebraically closed, we can assume that $r_i\overline{r_i}=1$. Since $(r_ix_i)^2=1$, $g_i^2=e$. Moreover, since $(r_1x_1)(r_2x_2)=(\overline{r_1}\overline{r_2})x_3$, we can assupe that $r_3=\overline{r_1}\overline{r_2}$ and that $g_3=g_1g_2$. Hence, $\Gamma$ is isomorphic to  $\Gamma_{SQ}(G,g,g_1,g_2)$ via the morphism $\theta(r_1,r_2,Id)$ with the notation of $\ref{l:AutGrad}$. 
	
\end{proof}

\begin{proposition}
	If $\Gamma\colon\cA=\bigoplus_{g\in G}\cA_g$ is a grading compatible with $\Gamma^i_{S}$ for some $i=1,2,3$ but not with $\Gamma_{SQ}$ then either it is isomorphic to $\Gamma_{S}(G,\lambda,h,g,f)$ with elements as in example \ref{ex:Gradings} or it is isomorphic to $\Gamma^i_{S}(G,h,g)$ for $i=1,2$ with the notation as in example \ref{ex:Gradings}.
\end{proposition}

\begin{proof}
	Up to isomorphism we can suppose that it is compatible with $\Gamma_{S}^1$. Due to lemma \ref{l:KandMGraded}, $\cK x_1$ is a graded subspace.
	
	If $\deg(s)=e$, since $s\cK x_1=\cK x_1$, then $e_+x_1$ and $e_-x_2$ are homogeneous. Let $\deg(e_+x_1)=h$. Since $(e_+x_1)(e_-x_1)$, we get that $\deg(e_- x_1)=h^{-1}$. We are going to prove that $e_+ x_2$ cannot be homogeneous. We prove it by contradiction. If it is homogeneous of degree $g$, $(e_+x_2)(e_+x_1)=e_-x_3$ is homogeneous of degree $gh$. Necesarily, there should be a $\lambda\in\FF$ such that $e_+(\lambda x_2+x_3)$ is homogeneous of degree $f$. Necessarily $f\neq g$ otherwise $e_+x_3$ is homogeneous and multiplying by $e_+x_1$, we get that $e_-x_1$ is homogeneous and that means that the grading is compatible with $\Gamma_{SQ}$. Now, multiplying by $e_+x_1$ we get that $e_(x_2+\lambda x_3)$ is homogeneous of degree $fh$. Since $b(e_-(x_2+\lambda x_3),e_+x_2)=\frac{1}{2}$ and $b(e_-(x_2+\lambda x_3),e_+(\lambda e_2,e_3))=\lambda$ and since $\ker(b(e_-(x_2+\lambda x_3)),\ \cdot \ )_{\mid (\cK x_2\oplus \cK_{x_3})\cap \cM_+}$ has to be a graded subspace, then $\lambda=0$. Therefore, $e_+x_3$ is homogeneous and as we saw before, this leads to a contradiction with the fact that $\Gamma$ is not compatible with $\Gamma_{SQ}$.
	
	Due to the previous discussion, we can assume (because we can multiply by scalar) that there are $\lambda,\beta\in\FF^{\times}$ such that $e_+(x_2+\lambda x_3)$ is homogeneous of degree $g$ and $e_+(\beta x_2+ x_3))$ is homogeneous of degree $f$ and both are linearly independent. Multiplying by $e_+x_1$ you get that $e_-(\lambda x_2+x_3)$ is homogeneous of degree $gh$ and that $e_-(x_2+\beta x_3)$ is homogeneous of degree $fh$. Call $\varphi=b(e_-(\lambda x_2+x_3),\ \cdot \ )_{\mid (\cK x_2\oplus \cK_{x_3})\cap \cM_+}$. Since $\ker(\varphi$ has to be a graded subspace of $x_2\oplus \cK_{x_3})\cap \cM_+$ and $\varphi(e_+(x_2+\lambda x_3))=\lambda\neq 0$, necessarily, $0=\varphi(e_+(\beta x_2+x_3 ))=\frac{1}{2}(\lambda \beta +1)$. In order to see that $g^2=f^2=h^{-1}$, we see that the square of $(e_+(x_2+\lambda x_3))$ and of $e_+(-\lambda^{-1}x_2+x_3)$ are nonzero multiples of $e_-x_1$. Therefore, $\beta=-\lambda^{-1}$ and therefore, $\Gamma$ is isomorphic to $\Gamma_{S}(G,\lambda,h,g,f)$.
	
	If $\deg(s)=h$ for $h\neq e$, clearly $h^2=e$. By lemma \ref{l:MpmGraded} we know that there is an invertible $r_1\in \cK$ such that $r_1x_1$ is homogeneous of degree $f$. Using the automorphism $\theta(\frac{1}{\sqrt{r_1\overline{r_1}}}\overline{r_1},1,Id)$ we can assume that $r_1=1$. We are going to prove by contradiction that  for no $r_2\in\cK$, $r_2x_2$ is homogeneous. Suppose it is. If $r_2$ is multiple of $e_{\sigma}$ for some $\sigma=\pm$, then $h\deg(r_2x_2)=\deg(s(r_2x_2))=\deg(r_2x_2)$ which would be a contradiction.  If $r_2$ is invertible, since $s(r_2 x_2)=(sr_2)x_2$, $x_1(r_2x_2)=\overline{r_2}x_3$ and $(s\overline{r_2})x_3$ are homogeneous, $\Gamma$ would be compatible with $\Gamma_{SQ}$. Hence, there are $r_2,r_3\in\cK\setminus {0}$ such that $r_2x_2+r_3x_3$ is homogeneous of degree $g$, then multiplying by $x_1$ we get that $\overline{r_3}x_2+\overline{r_2}x_3$ is homogeneous of degree $gf$. 
	
	If there is no $r_2x_2+r_3x_3$ homogeneous with $r_2,r_3$ invertible, then for an homogeneous element like this, $(r_2x_2+r_3x_3)^2=2 \overline{r_2r_3}x_1=0$ since $r_2r_3\in \FF e_{\sigma}$ for some $\sigma=\pm$. We can suppose then that $r_2=\lambda e_{\sigma}$ and $r_3=\beta e_{-\sigma}$. Moreover, by scaling the element we can suppose that $\lambda=1$. Call $x=r_2x_2+r_3x_3$. Since $x_1x$ has is a linear combination of $x$ and $sx$ then, either $x_1x=x$ in which case $\beta=1$ and $\deg(x_1)=e$ or $x_1x=sx$, in which case $\beta=-1$ and $\deg(x_1)=\deg(s)$. Using $\theta(s,1,Id)$ if necessary, we can suppose that $\beta=1$ and $\deg(x_1)=e$. Hence, $x=e_+x_2+e_-x_3$ is homogeneous of degree $g$ and since there should be an homogeneous element which doesn't belong to $\mathrm{span}\{x,sx,x_1x,(sx_1)x\}=\mathrm{span}\{x,sx\}$, using the same arguments we see that $y=e_-x_2+e_+x_3$ is homogeneous. Since $xy=1+x_1$, we get that $b(x,y)\neq 0$ and so $y$ is homogeneous of degree $g^{-1}$. Hence, $\Gamma$ is isomorphic to $\Gamma^1_{S}(G,h,g)$
	
	Finally, assume that $r_2x_2+r_3x_3$ is homogeneous with $r_2,r_3$ invertible, using the automorphism $\theta(1,\frac{1}{\sqrt{r_2\overline{r_2}}}\overline{r_2})$ and multiplying by scalar we can assume that $r_2=1$. Since $(x_2+r_3x_3)^2=2\overline{r_3}x_1$ and it is homogeneous, we can assume that $r_3\in\FF 1\cup \FF s$. Using if necesary the automorphism $\theta(s,1,\id)$ we can suppose that $r_3=\lambda 1$ for some $\lambda\in\FF^{\times}$. If $\deg(x_1)\neq e$,  since $b(x,x_1x)\neq 0$ we get that $b(x,x_1x)=0$ and that means that $\lambda^2=-1$. Since in this case $-\lambda x_1x=x_2-\lambda x_3$, any choice of $\lambda$ would be an homogeneous element. Hence the grading is isomorphic to $\Gamma_{S}^2(G,h,g)$. Finally, if $\deg(x_1)=e$, $\lambda=\pm 1$. Since $\cK x+\cK(x_1x)=\FF x\oplus \FF sx$ we need to complete with another homogeneous element. By the same argument it has to be $y=x_2-x_3$ so the grading is $\Gamma_{S}^3(G,h,g,f)$ where $\deg(x_2+x_3)=g$.
\end{proof}

\begin{proposition}
Let $\Gamma\colon\cA=\bigoplus_{g\in G} \cA_g$ be a grading on $(\cA,\invol)$ which is not compatible with any $\Gamma^i_{S}$. Then it is isomorphic either to $\Gamma(G,g_1,g_2)$ for $g_1,g_2$ of order $3$ or to $\Gamma(G, h ,g)$ for $h$ of order $2$ and $g$ of order $3$.
\end{proposition}

\begin{proof}
	If $\deg(s)=e$ we start by proving that $e_{\sigma}x_i$ cannot be homogeneous. Since we can use the automorphisms $f_{\tau}$ and $\theta(1,1,\ex)$, we can prove it for $i=1$ and $\sigma=+$. In this case, $\ker(b(e_+x_1,\ \cdot \ ))\cap\cM_{-}=\FF e_-x_2\oplus \FF e_-x_3$ is homogeneous. Then, since $b$ is non degenerate, there should be an homogeneous element of degree $g$, $x=e_-(x_1+\lambda_2x_2+\lambda_3 x_3)$ with $\lambda_2,\lambda_3\in \FF$. Since $b(e_+x_1,x)$ and $b(x^2,x)$ are not $0$, it follows that $e_+x_1$ and $x^2$ have the same degree. Since $y=\frac{1}{2}x^2-\lambda_2\lambda_3e_+x_1=e_+(\lambda_3x_2+\lambda_2x_3)$, if $\lambda_2\lambda_3\neq 0$ then, since $y^2$ is homogeneous, then $e_-x_1$ is homogeneous and then $\ker(b(e_+x_1,\ \cdot \ ))\cap\ker(b(e_-x_1,\ \cdot \ ))=\cK x_2\oplus \cK x_3$ is graded and because of that this grading is compatible with $\Gamma_{S}^1$. If $\lambda_2\neq 0$ but $\lambda_3= 0$. Since $x^2(e_+x_1)$ is homogeneous, $e_-x_2$ is homogeneous. Hence $z=x(e_-x_2)=e_+(\lambda_3x_1+x_3)$ is homogeneous. Since $b(x,z)=\lambda_3\neq 0$ it follows that $z$ and $e_+x_1$ have the same degree and so $z-\lambda_3 e_+x_1=e_+x_3$ is homogeneous. Therefore, $(e_+x_2)(e_+x_3)=e_-x_1$ is homogeneous and it follows as before that it is not compatible with $\Gamma_{S}^1$. If $\lambda_2=\lambda_3=0$ we have it because of the same argument.
	
	Let $x=e_+(\lambda_1x_1+\lambda_2x_2+\lambda_3x_3)$ be an homogeneous element of degree $g$. It follows that $\lambda_1\lambda_2\lambda_3\neq 0$. Hence, by scalar multiplication we can assume that $\lambda_1\lambda_2\lambda_3=1$. Take another homogeneous element $y=e_+(\beta_1x_1+\beta_2x_2+\beta_3x_3)$ of degree $h$ with $\beta_1\beta_2\beta_3= 1$ such that $g\neq h$ (which should exist since the grading is not compatible with $\Gamma_{S}^1$). We can check that $(x^2)^2=4x$ and $(y^2)^2=4y$. That means that $g^3=e$ and $h^3=e$. Moreover, its easy to see that $b(x,x^2)=6 \lambda_1\lambda_2\lambda_3\neq 0$ and $b(y,y^2)=6\beta_1\beta_2\beta_3$. And because $h^2g\neq e$ we deduce that $b(x,y^2)=b(y, x^2)=0$. That implies that \begin{equation}\label{eq:lambdaBeta}
		\lambda_1\lambda_2\beta_3+\lambda_2\lambda_3\beta_1+ \lambda_3\lambda_1\beta_2=\beta_1\beta_2\lambda_1+\beta_2\beta_3\lambda_1+\beta_3\beta_1\lambda_2=0
	\end{equation} 
Moreover, since $hg\neq g^{2}$ and $gh\neq h^2$, we deduce that $xy\neq x^2$ and $xy\neq y^2$. Since $xy=e_+[(\lambda_2\beta_3+\lambda_3\beta_2)x_1+ (\lambda_1\beta_3+\lambda_3\beta_1)x_2+(\lambda_2\beta_1+\lambda_1\beta_2)x_3]$ using \eqref{eq:lambdaBeta} we see that $xy=e_+(-\lambda_2\lambda_3\beta_1\lambda_1^{-1}x_1-\lambda_1\lambda_3\beta_2\lambda_2^{-1}x_1-\lambda_1\lambda_2\beta_3\lambda_3^{-1}x_3)$. Therefore, if we call $z'=-xy$ we can see that it's coefficients products equals to $1$. Hence, for $z=\frac{1}{2}{z'}^2 $ we get that $z^2=2z`$ and $(z^2)^2=4z$. And we can check that the map sending $x\mapsto \deg(e_+(x_1+\zeta x_2+\zeta^2 x_3))=g_1$, $y\mapsto\deg(e_+(x_1+\zeta^2x_2+\zeta x_3))=g_2$ and  $z\mapsto\deg(e_+(x_1+x_2+x_3))=(g_1g_2)^{-1}$ is an isomorphism and so the grading is isomorphic to $\Gamma(G,g,h)$.

If $\deg(s)=h$, as before, $\cK x_i$ cannot be a graded subspace.

If all the homogeneous elements  $x=r_1x_1+r_2x_2+r_3x_3$ such that $r_1,r_2$ and $r_3$ are non zero, then, the projection of $x^2$ in $\cM$ is $y=\overline{r_2r_3}x_1+\overline{r_1r_3}x_3+\overline{r_1r_2}x_3$ which is homogeneous. If $r_1$ and $r_2$ are not invertible,  then this is in $\cM_{\sigma}$ for $\sigma=\pm$ and it would happen that $\deg(y)=\deg(sy)$ which can't happen unless $r_3=0$. If $r_1$ is not invertible but $r_2$ and $r_3$ are invertible, we use $y$ to show a contradiction. Hence $r_1,r_2$ and $r_3$ are invertible. Using the map $\theta(\frac{1}{\sqrt{r_1\overline{r_1}}}\overline{r_1},\frac{1}{\sqrt{r_2\overline{r_2}}}\overline{r_2})$ we can assume that $r_1$ and $r_2$ are scalars. If $\deg(x)=g$, since $(x^2)^2=(r_1r_2r_3)x$, we can assume that either $r_3\in\FF s$ or $r_3\in \FF 1$. Since we can scale we can suppose that $r_1r_2 r_3=1$ or $r_1r_2r_3=s$. In the first case and in the second case $g^3=e$ $g^3h=e$ if $g^3h=e$ we can multiply by $s$ and use the automorphism $\theta(s,s,\id)$ and we are in the first case. Now, either there is an element like this whose degree has order $3$ or there are $3$ linearly independent elements whose degree is $3$. In the second case, necessarily, since $\cK$ has dimension $2$, there must be an element of degree $3$ such that $r_1,r_2$ or $r_3$ is $0$ so we don't consider it here. Now, since the projections of  $x,x^2, x^2x,sx,s(x^2x)$ on $\cM$ span $\cM$, necessarily, $b(x,x^2)=6r_1r_2r_3\neq 0$ and that implies $r_3\in\FF 1$. Now, up to scalar, we can suppose that $r_1r_2r_3=1$ and we can check that the map sending $x\to x_1+\zeta x_2+\zeta x_3$ induces an isomorphism of algebras. Therefore, $\Gamma$ is isomorphic to $\Gamma(G,h,g)$.

Finally, we will show that these are all the possibilities. Indeed, if there is an homogeneous element $x=r_1x_1+r_2x_2$ of degree g for $r_1$ and $r_2$ different from $0$, since $x^2=2\overline{r_1r_2}x_3$ necessarily, we get that $r_1r_2=0$. Hence, we can suppose that there is $\lambda_1,\lambda_2\in\FF$ such that $x=\lambda_1e_+x_1+\lambda_2e_-x_2$. Moreover, $sx=\lambda_1e_+x_1-\lambda_2e_-x_2$ is also homogeneous of degree $gh$. We can show that all homogeneous elements $y=t_1x_1+t_2x_2+t_3x_3$ with $t_1,t_2,t_3\in\cK$ have $t_1,t_2$ or $t_3$ equal to $0$. Otherwise, $xy=e_-\overline{t_2}+e_+\overline{t_3}+(e_+\overline{t_3})x_1+(e_- \overline{t_3})x_2+(e_-\overline{t_2}+e_+\overline{t_3})x_3$ so either $y$ or $xy$ has coefficients which are not invertible and arguing as before, this is impossible. Hence, all the homogeneous elements in $\cM$ should be of the form $\lambda_ie_{\sigma}x_i\pm\lambda_j e_{-\sigma}x_j$ for $i\neq j$ and $\lambda_i,\lambda_j\in \FF$. We can finally show, that if $x=\lambda_1e_+x_1+\lambda_2e_-x_2$ is homogeneous, there should be $\beta_2,\beta_3\in\FF^{\times}$ such that $y=\beta_2e_++\beta_3e_-$ is homogeneous. But since $xy=\lambda_2\beta_2 e_-+\lambda_1\beta_2 e_-x_3+\lambda_2\beta_3 e_+x_1$ and that would imply that $e_-$ is homogeneous since $\cK$ and $\cM$ are homogeneous subspaces. But this would be a contradiction with the fact that $\deg(s)\neq \deg(1)$.
\end{proof}

\end{document}